\documentclass[12pt,a4paper]{amsart}
\usepackage{graphicx} % Required for inserting images

\usepackage{fancyhdr,mathtools}
\usepackage{enumitem}
\usepackage{appendix}
\usepackage{amssymb,amscd,amsxtra,calc}
\usepackage{mathrsfs}
\usepackage{multirow}
\usepackage[all]{xy}
\usepackage{longtable}
\usepackage[colorlinks,linkcolor=blue,anchorcolor=red,citecolor=green,linktocpage, pagebackref]{hyperref}
\usepackage{xcolor}

\usepackage{mathabx,epsfig}
\def\acts{\ \rotatebox[origin=c]{-90}{$\circlearrowright$}\ }
\def\racts{\ \rotatebox[origin=c]{90}{$\circlearrowleft$}\ }

\theoremstyle{plain}
    \newtheorem{thm}{Theorem}[section]

       \newtheorem{lemma}[thm]{Lemma}
           \newtheorem{theorem}[thm]{Theorem}
            
    \newtheorem{proposition}[thm]{Proposition}

\theoremstyle{definition}
 
    \newtheorem{claim}[thm]{Claim}
    \newtheorem{definition}[thm]{Definition}
    
     \newtheorem{corollary}[thm]{Corollary}
     
       \newtheorem{conjecture}[thm]{Conjecture}
    \newtheorem{question}[thm]{Question}
    
    \newtheorem{notation}[thm]{Notation}
    \newtheorem*{notation*}{Notation and Terminology}
    \newtheorem{remark}[thm]{Remark}
    
\theoremstyle{remark}

\numberwithin{equation}{section}

\newcommand{\bbr}{\mathbb{R}}
\newcommand{\bbq}{\mathbb{Q}}
\newcommand{\bbz}{\mathbb{Z}}

\newcommand{\cali}{\mathcal{I}}

\renewcommand{\ker}{\textup{Ker}\ }

\newcommand{\Q}{\mathbb{Q}}

\newcommand{\R}{\mathbb{R}}

\newcommand{\Z}{\mathbb{Z}}

\newcommand{\OO}{\mathcal{O}}

\newcommand{\diag}{\operatorname{diag}}

\newcommand{\id}{\operatorname{id}}
\newcommand{\Imm}{\operatorname{Im}}

\newcommand{\NE}{\overline{\operatorname{NE}}}
\newcommand{\Nef}{\operatorname{Nef}}

\newcommand{\NS}{\operatorname{NS}}

\newcommand{\PE}{\operatorname{PE}}

\newcommand{\rank}{\operatorname{rank}}

\newcommand{\Supp}{\operatorname{Supp}}

\newcommand{\N}{\operatorname{N}}

\newcommand{\Pic}{\operatorname{Pic}}

\title[Dynamical Iitaka theory]{Dynamical Iitaka theory on Fano contractions}

\author{Sheng Meng}
\address{
\textsc{School of Mathematical Sciences, Shanghai Key Laboratory of PMMP}\endgraf 
\textsc{East China Normal University, Shanghai 200241, China}
}
\email{smeng@math.ecnu.edu.cn}

\author{Long Wang}
\address{Center for Mathematics and Interdisciplinary Sciences, Fudan University, Shanghai, 200433,
China; Shanghai Institute for Mathematics and Interdisciplinary Sciences, Shanghai, 200433, China}
\email{wanglll@fudan.edu.cn}

\author{Tianle Yang}
\address{
\textsc{School of Mathematical Sciences, Shanghai Key Laboratory of PMMP}\endgraf 
\textsc{East China Normal University, Shanghai 200241, China}
}
\email{10211510098@stu.ecnu.edu.cn}



\thanks{The first author is supported by the Shanghai Pilot Program for Basic Research, the Science and Technology Commission of Shanghai Municipality (No. 22DZ2229014), and the National Natural Science Foundation of China. The second author is supported by the National Key Research and Development Program of China (No. 2023YFA1010600), the National Natural Science Foundation of China (No. 12401052), and the NSFC for Innovative Research Groups (No. 12121001).} %Funding information 

\begin{document}

\begin{abstract}
    We give several structure theorems for certain surjective endomorphisms on Mori fibre spaces,
    based on the dynamical Iitaka fibration of the ramification divisor.
    As an application, we prove the Kawaguchi-Silverman conjecture for projective bundles over abelian varieties or smooth projective varieties of Picard number one.
\end{abstract}

\maketitle

\tableofcontents

\thispagestyle{empty}

%%%%
%%%%

\section{Introduction}\label{sec: introduction}

We work over an algebraically closed field $\mathbf{k}$ of characteristic zero, unless specified otherwise.

\subsection{Dynamical Iitaka dimension and fibration}\label{subsec: background}
The study of surjective endomorphisms on smooth projective varieties lies at the intersection of algebraic geometry and complex dynamics, 
with central questions revolving around their structural decomposition, invariance and equivariancy. 
Precisely, given a surjective endomorphism \( f: X \to X \) of a normal projective variety $X$, we aim to find possible dominant rational maps $\pi:X\dashrightarrow Y$ such that there exists a surjective endomorphism $g\colon Y\to Y$ with $g\circ\pi=\pi\circ f$. 
When such $g$ exists, we say $\pi$ is $f$-equivariant and will use $f|_Y$ instead of $g$ to make the notation clear.

In \cite{MZ23b}, De-Qi Zhang and the first author introduced a natural concept of the dynamical Iitaka dimension and dynamical Iitaka fibration. It says that given any effective $\mathbb{Q}$-Cartier divisor $D$, 
one can construct an $f$-equivariant dominant rational map 
$$\varphi_{f,D}\colon X\dashrightarrow Y$$ such that
$\dim Y =\kappa_f(X,D)$
where $\kappa_f(X,D)$ is called the $f$-Iitaka dimension of $D$.
We refer to Section \ref{sec: dynamical iitaka} for the precise definitions and several facts.

For simplicity, we now assume $X$ is smooth.
The dynamical Iitaka theory offers a powerful framework that enables us to construct equivariant fibrations. 
One of its remarkable features is the ability to initiate this construction from an arbitrary effective divisor $D$, 
breaking free from the constraint that the divisor $D$ must be a pullback eigenvector. 
In \cite{MZ23b}, several progress has been made in exploring the application of this theory to exceptional divisors of threefolds.
However, we contend that the study of the ramification divisor \(R_f\) holds even greater significance. 
When the variety \(X\) lacks a divisorial contraction, \(R_f\) emerges as a central object of investigation. 
Unlike other divisors, \(R_f\) is inherently linked to the dynamics of surjective endomorphisms, and its effectiveness provides a solid foundation for further analysis. 
By focusing on \(R_f\) within the dynamical Iitaka theory, 
we can delve deeper into the intrinsic geometric and dynamical structures of \(f\acts X\), 
potentially uncovering new insights that are otherwise obscured when relying solely on the study of other types of divisors. 
This makes the exploration of \(R_f\) not only a natural but also an essential direction for advancing our understanding of the dynamical and geometric properties of varieties in the absence of divisorial contractions. 

\smallskip
\noindent\textbf{Dynamical Iitaka Program for ramification divisor.}
We consider three situations of $\kappa_f(X,R_f)$ and briefly introduce how the dynamical Iitaka theory works.
Here $\kappa_f(X,R_f)$ denotes the dynamical Iitaka dimension of $R_f$, which is exactly the Iitaka dimension of $R_{f^t}$ when $t$ is sufficiently large. We refer to Definition \ref{def: dynamical Iitaka dimension} and Lemma \ref{lem: composition formula} for more details. 
%$\kappa_f(X,R_f) = \kappa(X,R_{f^t})$ for $t \gg 0$ (see Definition \ref{def: dynamical Iitaka dimension} and Lemma \ref{lem: composition formula} for more details).
\begin{itemize}
    \item $\kappa_f(X,R_f)=0$. 
    Then $f^{-1}(\Supp R_f)=\Supp R_f$ and the restriction $$f|_{X\backslash \Supp R_f}\colon X\backslash \Supp R_f\to X\backslash \Supp R_f$$ is \'etale.
    Certain cases have been studied;
    see \cite{MZ19, MZg23, MZg24, MYY25, MYY24b}.
    \item $0<\kappa_f(X,R_f)<\dim X$.
    We then have an $f$-equivariant dominant rational map
    $$\varphi_{f,R_f}\colon X\dashrightarrow Y$$
    with $0<\dim Y=\kappa_f(X,R_f)<\dim X$.
    One central problem is stated in Question \ref{que: main1} that whether the first dynamical degree is preserved, i.e., $\delta_f=\delta_{f|_Y}$. Here, $\delta_f$ is the spectral radius of $f^*|_{\NS(X)}$ and $\delta_f\ge \delta_{f|_Y}$ always holds; see Definition \ref{def: delta_f}.
    \item $\kappa_f(X,R_f)=\dim X$. 
    In this case, $R_{f^s}$ is big when $s\gg 1$ (Lemma \ref{lem: composition formula}).
    De-Qi Zhang and the first author proposed Question \ref{que: main2}. 
    A positive answer to this question will allow us to apply the equivariant minimal model program to further study $f$; see \cite{MZ23a} for a survey.
\end{itemize}

\begin{question}\label{que: main1}
    Let $f\colon X\to X$ be a surjective endomorphism of a smooth projective variety with $\varphi_{f,R_f}\colon X\dashrightarrow Y$ the induced dynamical Iitaka fibration.
    Suppose $\kappa_f(X, R_f)>0$. Does the equality of first dynamical degrees $\delta_f=\delta_{f|_Y}$ hold?
\end{question}

\begin{question}[{\cite[Question 6.9]{MZ23b}}]\label{que: main2}
    Let $f\colon X\to X$ be a surjective endomorphism of a smooth projective variety.
    Suppose $R_f$ is big. Is $f$ int-amplified?
\end{question}

\subsection{Structure theorems}
In this paper, we mainly develop the dynamical Iitaka theory on one fundamental situation appeared in the minimal model program: when $X$ admits an extremal Fano contraction. 

The following are our main results.
Note that we verify Question \ref{que: main1} and \ref{que: main2} in the situation of Theorem \ref{mainthm: abelian} and \ref{mainthm: rho1}.
By extremal Fano contraction (or the Mori fibre space), we mean a contraction of some extremal ray with the dimension decreased.

\begin{theorem}\label{mainthm: abelian}
    Let $f$ be a surjective endomorphism of a smooth projective variety $X$ admitting an extremal Fano contraction $\pi:X\to Y$ to an abelian variety $Y$ of positive dimension.
    Suppose $f$ admits a Zariski dense orbit and $\delta_f>\delta_{f|_Y}$.
    Then the following hold.
    \begin{enumerate}
        \item The ramification divisor satisfies $f^*R_f\equiv \delta_f R_f$.
        \item There exists an $f$-equivariant dominant rational map $\varphi:X\dashrightarrow Z$, 
        which is the $f$-Iitaka fibration of $R_f$, 
        such that $0<\dim Z<\dim X$ and $f|_Z$ is $\delta_f$-polarized.
    \end{enumerate}
\end{theorem}

\begin{remark}
  The assumption of $f$ admitting a Zariski dense orbit is natural.
  Indeed, if $f$ does not admit such an orbit, the Zariski dense orbit conjecture asserts the existence of an $f$-equivariant dominant rational map to $\mathbb{P}^1$ and $f$ descends to an automorphism of finite order. 
  The Zariski dense orbit conjecture is known to be true when the base field is further uncountable by Amerik and Campana \cite{AC08}.
  We refer to \cite{Xie25} for more information.
\end{remark}

When the base of the Fano contraction is a smooth projective variety of Picard number one, we need to further require the contraction to be smooth and $\delta_{f|_Y}=1$ but relax the assumption on the existence of a Zariski dense orbit.
The $f$-equivariant assumption is automatically satisfied after iteration (cf.~\cite[Remark 6.3]{MZ18}).

\begin{theorem}\label{mainthm: rho1}
    Let $f$ be a surjective endomorphism of a smooth projective variety $X$ admitting an $f$-equivariant {\bf smooth} extremal Fano contraction $\pi:X\to Y$ with $\rho(Y)=1$.
    Suppose $\delta_f>\delta_{f|_Y}=1$.
    Then the following hold.
    \begin{enumerate}
        \item The ramification divisor satisfies $f^*R_f\equiv \delta_f R_f$.
        \item There exists an $f$-equivariant dominant rational map $\varphi:X\dashrightarrow Z$, 
        which is the $f$-Iitaka fibration of $R_f$, 
        such that $0<\dim Z<\dim X$ and $f|_Z$ is $\delta_f$-polarized.
    \end{enumerate}
\end{theorem}

A key point in the proof of Theorems \ref{mainthm: abelian} and \ref{mainthm: rho1} is the following,
which, as a relative version of \cite[Theorem 1.2]{MZ19}, provides a criterion for characterizing toric fibrations.
Here, by a \emph{Fano fibration}, we simply mean the general fibre is a Fano variety.

\begin{theorem}\label{mainthm: a character of toric fibraion over k}
    Let $f$ be a surjective endomorphism of smooth projective variety $X$.
    Let $\pi:X\to Y$ be an $f$-equivariant Fano fibration.
    Suppose there is a reduced divisor $D$ such that $K_X+D$ is $\pi$-trivial and $f^*D=qD$ for some $q>1$.
    Then $(X_y, D|_{X_y})$ is a toric pair for general $y\in Y$.
\end{theorem}

\subsection{Kawaguchi-Silverman conjecture} 
Let $f \colon X \to X$ be a surjective endomorphism. %defined over $\overline{\bbq}$.
The first dynamical degree $\delta_f$ measures the geometric complexity of the dynamical system $(X,f)$. When $(X,f)$ is defined over $\overline{\bbq}$, we can also define an arithmetic complexity function pointwisely, that is, the arithmetic degree $\alpha_f(x)$ (see Definition \ref{def: alpha_f}). 
Roughly speaking, $\alpha_f(x)$ reflects the growth of the Weil height function along the (forward) $f$-orbit $$O_f(x) := \{f^n(x)\, |\, n\ge 0\}.$$ 
A conjecture of Kawaguchi and Silverman (\cite{KS16b}) comparing the dynamical and the arithmetic degrees for endomorphisms can be stated as follows.
One should notice that the original conjecture is formulated for dominant self-maps.
% For a surjective endomorphism $f: X \to X$ defined over $\overline{\bbq}$, 
% except for the first dynamical degree $\delta_f$ which is measuring the geometric complexity of the dynamical system $(X, f)$, 
% one can also define an arithmetic complexity function pointwisely, that is, the arithmetic degree $\alpha_f(x)$ (see Definition \ref{def: alpha_f}). 
% Roughly speaking, $\alpha_f(x)$ reflects the growth of the Weil height function along the (forward) $f$-orbit $$O_f(x) := \{f^n(x)\, |\, n\ge 0\}.$$ 
% A conjecture of Kawaguchi and Silverman (\cite{KS16b}) comparing the dynamical and the arithmetic degrees for endomorphisms can be stated as follows.
% One should notice that the original conjecture is formulated for dominant self-maps.

\begin{conjecture}[{Kawaguchi-Silverman Conjecture $=$ KSC}]\label{conj: ksc} 
    Let $f : X \to X$ be a surjective endomorphism of a projective variety $X$ defined over $\overline{\mathbb{Q}}$. 
    Let $x \in X(\overline{\mathbb{Q}})$, and suppose %the (forward) orbit $O_f(x) = \{f^n(x)\, |\, n\ge 0\}$ 
    that $O_f(x)$ is Zariski dense in $X$.
    Then the arithmetic degree at $x$ is equal to the dynamical degree of $f$, i.e.,
    $\alpha_f(x) = \delta_f$.
\end{conjecture} 

Let us digress to mention that, the arithmetic degree, or the growth of height functions along the orbit, also helps to solve other arithmetic problems on dynamics, such as the Zarski dense orbit conjecture (\cite{JSXZ24, MW24, MX24, Wa24}), and the dynamical Mordell-Lang conjecture (\cite{Xie23survey, Xie23, XY25b, XY25a}). This demonstrates the fundamentality of the arithmetic degree, and the importance of the Kawaguchi-Silverman conjecture as it helps to understand the arithmetic degree. 

The Kawaguchi-Silverman conjecture is still widely open, although some special cases were worked out over the years, for which we refer to \cite{Matz23} and \cite{MZ23a} (see also Section \ref{subsec: known_results_ksc}).

We use our structure theorems to study the Kawaguchi-Silverman conjecture, and provide further evidence supporting this conjecture. It is worthy to notice that, except for its own interest, the dynamical Iitaka theory is also established for some arithmetic problems of dynamics \cite{MZ23b}, such as the Kawaguchi-Silverman conjecture and the Zariski dense orbit conjecture.  

\begin{corollary}\label{maincor: abelian}
    KSC holds for any smooth projective variety $X$ admitting an extremal Fano contraction to an abelian variety. 
\end{corollary}

\begin{corollary}\label{maincor: pn}
    KSC holds for any $\mathbb{P}^n$-bundle over either a $Q$-abelian variety or a smooth projective variety of Picard number one. 
\end{corollary}

Corollary \ref{maincor: pn} generalizes \cite[Theorem 4.1]{LM21}, where the authors proved the case of projective bundles on smooth Fano varieties of Picard number one. Either Corollary \ref{maincor: abelian} or Corollary \ref{maincor: pn} generalizes \cite[Theorem 1.4]{NZ23}, where the authors proved the case of projective bundles on elliptic curves. Furthermore, compared with the proof of \cite{NZ23} using the classification of vector bundles on elliptic curves, our argument is conceptual.

\subsection{New ingredients and comments} 
As we mentioned, one motivation to develop the dynamical Iitaka theory is to solve the Kawaguchi-Silverman conjecture. 
In view of Lemma \ref{lem: ksc iff}, it is natural to consider the case $\delta_f > \delta_{f|_Y}$.

This paper based on and developed the ideas of \cite{MZ23b}. There are some new ingredients on dealing with an equivariant extremal Fano contraction $(X,f) \xrightarrow{\pi} (Y,g)$ with $\delta_f > \delta_g$. The first one is Theorem \ref{mainthm: a character of toric fibraion over k}. To put it simply, we show that if $\pi$ is an extremal Fano contraction with $\Supp R_f$ $f^{-1}$-invariant, then $(X_y,\Supp R_f \cap X_y)$ is a toric pair for a general fibre $X_y$ of $\pi$. There, we carefully deal the relative version of semistability of sheaves of reflexive logarithmic $1$-form. This result will play a key role showing the positivity of  $\kappa_f(X,R_f)$ as stated in Theorem \ref{thm: pos}. Another one is the decompositions of cones as stated in Theorem \ref{thm: cone}. Under certain natural conditions, we show the decompositions
\[ 
\begin{aligned}
    \Nef(X) &= \pi^{\ast}\Nef(Y) \oplus \mathbb{R}_{\geq 0} D \\ 
    \textup{PE}(X) &= \pi^{\ast}\textup{PE}(Y) \oplus \mathbb{R}_{\geq 0} D
\end{aligned} 
\] 
of the nef cone and the pseudo-effective cone, where $D$ is a $\delta_f$-eigenvector of $f^*|_{\N^1(X)}$. With this, we can further show that $R_f$ is proportional to this $D$, and is not big (see Section \ref{sec: proofs of structure theorems}). Here two versions of the decomposition theorem are given, although each of them is sufficient for our purpose. We hope that their variants or generalizations will be also useful to other problems.

%This paper based on and developed the ideas of \cite{MZ23b}. There are some new ingredients. The first one is Theorem \ref{mainthm: a character of toric fibraion over k}. Roughly speaking, we show that for an equivariant extremal Fano contraction $(X,f) \xrightarrow{\pi} (Y,g)$ of smooth projective varieties such that $\Supp R_f$ is $f^{-1}$-invariant, either $\delta_f=\delta_g$, or $(X_y,\Supp R_f \cap X_y)$ is a toric pair for general fibre $X_y$ of $\pi$. There, we carefully deal the relative version of semistability of sheaves of reflexive logarithmic $1$-form. This result will play a key role showing the positivity of  $\kappa_f(X,R_f)$ as stated in Theorem \ref{thm: pos}. Another one is the decompositions of cones as stated in Theorem \ref{thm: cone}. When $\delta_f > \delta_g$ and under certain conditions, we have the decompositions
%\[ \begin{aligned}
%    \Nef(X) &= \pi^{\ast}\Nef(Y) \oplus \mathbb{R}_{\geq 0} D \\ 
%    \textup{PE}(X) &= \pi^{\ast}\textup{PE}(Y) \oplus \mathbb{R}_{\geq 0} D
%\end{aligned} \] 
%of the nef cone and pseudo-effective cone, where $D$ is a $\delta_f$-eigenvector of $f^*|_{\N^1(X)}$. With this, we further show that $R_f$ is not big and indeed a $\delta_f$-eigenvector of $f^*|_{\N^1(X)}$.

%\Yang{}
Currently, the main obstacle to generalizing the two structure theorems is the singularities of equivariant covers %$\widehat{Y}$ 
appeared in the proof of Claim \ref{claim: base change}.
We expect that the further development of Dynamical Iitaka Program will ‌contribute to‌ understand the structure of endomorphisms on projective varieties, and then help to solve the Kawaguchi-Silverman conjecture as well as other related arithmetic problems on dynamics.

\subsection{Structure of the paper} 
In Section \ref{sec: preliminaries}, 
we first fix notations which are used throughout the paper. 
Then we review the theory of sheaves of reflexive logarithmic $1$-form, which is the main tool to establish Theorem \ref{mainthm: a character of toric fibraion over k} on characterization of toric pairs. 
We also recall the notions of dynamical and arithmetic degrees, together with some basic facts and known results. In Section \ref{sec: dynamical iitaka}, we review the dynamical Iitaka theory. 

In Section \ref{sec: decomposition}, we present the decomposition theorem of nef and pseudo-effective cones (Theorem \ref{thm: cone}). The proof of Theorem \ref{mainthm: a character of toric fibraion over k} is given in Section \ref{sec: character of toric fibration}. After studying the positivity of ramification divisors in Section \ref{sec: positivity of ramification divisor}, we devote the proofs of two structure theorems (Theorems \ref{mainthm: abelian} and \ref{mainthm: rho1}) in Section \ref{sec: proofs of structure theorems}. The applications to the Kawaguchi-Silverman conjecture (Corollaries \ref{maincor: abelian} and \ref{maincor: pn} ) are given in Section \ref{sec: proofs of corollaries}.

\vskip 2mm
\noindent

\textbf{Acknowledgements.} We thank Yohsuke Matsuzawa and De-Qi Zhang for comments and suggestions to improve this paper. 

%The first author is supported by the Shanghai Pilot Program for Basic Research, the Science and Technology Commission of Shanghai Municipality (No. 22DZ2229014), and the National Natural Science Foundation of China. The second author is supported by the National Key Research and Development Program of China (No. 2023YFA1010600), the National Natural Science Foundation of China (No. 12401052), and the NSFC for Innovative Research Groups (No. 12121001). 

%%%%
%%%%

\section{Preliminaries}\label{sec: preliminaries}

\subsection{Varieties and divisors}
Throughout this paper, we will follow the following notation.
We refer to \cite{KM98} for notation of singularities.

\begin{notation}\label{not: notation}
    Let $X$ be a normal projective variety of dimension $n$.
    We use the following notation throughout this paper unless otherwise stated. 
\begin{longtable}{p{2cm} p{11cm}}

$\Pic(X)$    &the group of Cartier divisors of $X$ modulo linear equivalence~$\sim$\\
$\Pic^0(X)$    &the neutral connected component of $\Pic(X)$\\
$\NS(X)$    &$\Pic(X)/\Pic^0(X)$, the N\'eron-Severi group\\
$\N^1(X)$    &$\NS(X)\otimes_{\Z}\mathbb{R}$, the space of $\R$-Cartier divisors modulo numerical equivalence $\equiv$\\
$\N_r(X)$   &the space of $r$-cycle with $\bbr$-coefficients modulo weak numerical equivalence~$\equiv_w$\\
$\Nef(X)$   & the cone of nef classes in $\N^1(X)$\\
$\PE(X)$  & the cone of pseudo-effective classes in $\N^1(X)$\\
$\Supp D$   & the support of effective divisor $D=\sum a_i D_i$ which is $\bigcup_{a_i >0} D_i$, where the $D_i$ are prime divisors\\
$\rho(X)$    & the Picard number of $X$, which is $\dim_{\R} \N^1(X)$
\\ 
$\kappa(X, D)$    & the Iitaka dimension of a Cartier divisor $D$ on $X$ 
\\ 
$\nu(X, D)$    & the numerical dimension of a \textbf{nef} Cartier divisor $D$ on $X$, i.e., the largest integer $\nu$ such that $D^\nu\not\equiv_w 0$ (equivalent to $D^\nu \not\equiv 0$)
\end{longtable}

Here, two $r$-cycles \(Z_1\) and \(Z_2\) are said to be \textit{weakly numerically equivalent} (denoted by \(Z_1\equiv_w Z_2\)) if for any Cartier divisors \(H_1,\cdots,H_{n-r}\), one has \((Z_1-Z_2)\cdot H_1\cdots H_{n-r}=0\).
Note that two \(\mathbb{R}\)-Cartier divisors $D_1\equiv D_2$ if and only if $D_1\equiv_w D_2$; see \cite[Definition 2.2 and Lemma 2.3]{MZ18}.
\end{notation}

\subsection{Sheaf of reflexive logarithmic 1-form}
\label{subsec: forms}
\begin{definition}[Sheaf of reflexive logarithmic 1-form]
Let \(D\) be an effective reduced divisor on a normal projective variety \(X\).
Let \(U\xhookrightarrow{i} X \) be a big open subset (whose complement in \(X\) is codimension \(\geq 2\)) of the pair \((X,D)\) such that \((U,D\cap U)\) is log smooth.
We denote by $\Omega_U^1(U,D\cap U)$ the locally free sheaf of the usual logarithmic differential 1-forms.
The reflexive sheaf \(\hat{\Omega}_X^1(\textup{log}\,D)\coloneqq i_* \Omega_U^1(U,D\cap U)\) is called the 
\textit{sheaf of reflexive logarithmic 1-forms}. 
\end{definition}

The sheaf of reflexive logarithmic differential 1-forms is well-behaved under the finite pullback.

\begin{lemma}\label{lem: pullback of differential}
    Let \(f\colon X\to Y\) be a finite surjective morphism between smooth varieties.
    Let \(D_Y\) be a reduced divisor on \(Y\) and $D_X=f^{-1}(D_Y)$.
    Suppose that \(f|_{X\backslash D_X}\colon X\backslash D_X\to Y\backslash D_Y\) is \'etale.
    Then $f^*\hat{\Omega}_Y^1(\textup{log}\,D_Y)\cong\hat{\Omega}_X^1(\textup{log}\,D_X)$.
\end{lemma}

\begin{proof}
    Note that $f$ is flat.
    By \cite[Proposition 1.8]{Har80},
    the pullback $f^*\hat{\Omega}_Y^1(\textup{log}\,D_Y)$ is again a reflexive sheaf.
    By \cite[Lemma 2.7]{DLB22}, 
    we have $f^*\hat{\Omega}_Y^1(\textup{log}\,D_Y)\cong\hat{\Omega}_X^1(\textup{log}\,D_X)$.
\end{proof}

\begin{proposition}[{\cite[Proposition 2.4]{Dol07}}]\label{prop: the second exact sequence}
    Let $X$ be a smooth variety and $D$ a normal crossing divisor on $X$.
    Suppose $Z \subset X$ is a smooth closed subvariety of $X$ intersecting transversally $D$.
    Then there is an exact sequence of locally free sheaves
    \[ 0 \to \cali/\cali^2 \to \Omega_X^1(\log D)|_Z \to \Omega_Z^1(\log D \cap Z) \to 0\]
    where $\mathcal{I}$ is the ideal sheaf of $Z$.
\end{proposition}

\begin{definition}[Toric varieties]
    A normal variety $X$ of dimension $n$ is a \textit{toric variety} if $X$ contains a {\it big torus} $T=(\mathbf{k}^*)^n$ as an (affine) open dense subset such that the natural multiplication action of $T$ on itself extends to an action on the whole variety. 
    In this case, let $B\coloneqq X\backslash T$, which is a divisor; 
    the pair $(X,B)$ is said to be a \textit{toric pair}. 
\end{definition}

\begin{remark}
    Given a toric pair \((X,D)\), it is known that 
    $$\hat\Omega_X^1(\textup{log}\,D)\cong {\OO}_X^{\oplus \dim X}$$ 
    and $K_X+D\sim 0$ (see e.g. \cite[Remark 4.6]{MZ19}).
    In particular, if $X$ is $\bbq$-factorial, the divisor $D$ has exact $\dim X + \rho(X)$ components.
\end{remark}

\subsection{Endomorphisms and dynamical invariants}

\begin{definition}\label{def-endo}
Let \(f\colon X\to X\) be a surjective endomorphism of a projective variety \(X\). 
\begin{enumerate}
\item We say that \(f\) is \textit{$q$-polarized} if $f^*H\sim qH$ for some ample Cartier divisor $H$ and integer $q>1$, or equivalently,  \(f^*|_{\textup{N}^1(X)}\) is diagonalizable with all the eigenvalues being of modulus \(q\) (see \cite[Proposition 2.9]{MZ18}).
\item We say that  $f$ is \textit{int-amplified} if $f^*L-L$ is ample for some  ample  Cartier divisor $L$, or equivalently, all the eigenvalues of \(f^*|_{\textup{N}^1(X)}\) are of modulus greater than 1 (see \cite[Theorems 1.1 and 3.3]{Men20}). 
Clearly, every polarized endomorphism is int-amplified.
\item A subset $D\subseteq X$ is said to be {\it $f^{-1}$-invariant} if $f^{-1}(D)=D$. 
\end{enumerate}
\end{definition}

We refer to \cite{MZ23a} for a quick survey of the recent progress on polarized and int-amplified endomorphisms.

\begin{definition}[First dynamical degree]\label{def: delta_f}
Consider a projective variety \(X\) and a surjective endomorphism $f:X\to X$.
The \textit{first dynamical degree} \(\delta_f\) of \(f\) is defined to be the following limit
\[
\delta_f\coloneqq\lim_{n\to\infty}((f^n)^*H\cdot H^{\dim X-1})^{1/n}\in\mathbb{R}_{\geq 1},
\]
where $H$ is an ample Cartier divisor.
It is known that the limit always exists and is independent of the choice of $H$  (see \cite{DS04,DS05}; cf.~\cite{Dan20}). 
It is also known that the first dynamical degree \(\delta_f\) coincides with the spectral radius of the induced linear operation \(f^*|_{\textup{N}^1(X)}\).
Note that \(\delta_{f^s}=(\delta_f)^s\) for any positive integer $s$, just because of the existence of the limit defining $\delta_f$. 
Note also that, for an $f$-equivariant dominant rational map $\pi\colon X\dashrightarrow Y$, we always have $\delta_f \geq \delta_{f|_Y}$, and the equality holds if $\pi$ is generically finite (see, for instance, \cite[Lemma 2.4]{MZ22}).
\end{definition}

\begin{lemma}\label{lem: integral}
    Let $f$ be a surjective endomorphism of a normal projective variety $X$.
    Let $\pi:X\to Y$ be an $f$-equivariant surjective morphism to a normal projective variety $Y$ with $\rho(X)=\rho(Y)+1$ and $\delta_f>\delta_{f|_Y}$.
    Then $\delta_f$ is an integer, and there exists a nef and $\pi$-ample Cartier divisor $D$ on $X$ such that $f^*D\sim_{\Q} \delta_f D$.
\end{lemma}

\begin{proof} 
    By applying the Perron--Forbenius Theorem (\cite{Bir67}) to $f^*|_{\Nef(X)}$, 
    there is a nef $\bbr$-Cartier divisor $D_0 \not\equiv 0$ %class $0\not\equiv D_0 \in \N^1(X)_{\bbr}$
    such that $f^*D_0 \equiv \delta_fD_0$.
    Let $\chi_1(T)$, $\chi_2(T)$ be the characteristic polynomials of $f^*|_{\NS(X)}$ and $f^*|_{\NS(Y)}$ relatively.
    Then $\chi_i(T) \in \bbz[T]$. 
    Note that $\chi_1(T)=(T-\delta_f)\chi_2(T)$. 
    It follows that $\delta_f \in \bbz$. 
 
    Now we may choose $D_0$ to be $\mathbb{Q}$-Cartier. 
    By \cite[Theorem 6.4 (1)]{MMSZZ22}, 
    there exists a  $\mathbb{Q}$-Cartier divisor $D \equiv D_0$ such that $f^*D\sim_{\Q} \delta_f D$ because $\delta_f>1$. 
    After multiplication, we may further require $D$ to be Cartier. 
    Since $\delta_f > \delta_{f|_Y}$, we know $D \notin \pi^{\ast}\N^1(Y)$. 
    Let $H$ be an ample Cartier divisor on $X$.
    Note that $\rho(X)=\rho(Y)+1$ and $D\not\in \pi^*\N^1(Y)$.
    Then $H-aD\in \pi^*\N^1(Y)$ for some $a>0$.
    So we have $D$ being $\pi$-ample.  
\end{proof}

\begin{definition}[Arithmetic degree]\label{def: alpha_f}
Consider a projective variety \(X\) over \(\overline{\mathbb{Q}}\) and a surjective endomorphism $f:X\to X$.
\begin{enumerate}
\item 
For  \(D\in\textup{N}^1(X)\) on \(X\), there is a height function \(h_D\colon X(\overline{\mathbb{Q}})\to\mathbb{R}\) associated to \(D\), satisfying certain geometric and arithmetic properties, which measures the arithmetic complexity of \(\overline{\mathbb{Q}}\)-points. 
Such a function is uniquely determined up to a bounded function. 
We refer to \cite[Part B]{HS00} (cf.~%\cite[Theorem 1.1.1]{Kaw06}
\cite[Section 1.1]{Kaw06}) for an introduction to
Weil’s height theory. 
\item Let \(h_H\geq 1\) be an absolute logarithmic Weil height function associated with an ample Cartier divisor \(H\).
Then for every \(x\in X(\overline{\mathbb{Q}})\), we define the \emph{arithmetic degree of \(f\) at \(x\)} by
\[
\alpha_f(x)=\lim_{n\to\infty}h_H(f^n(x))^{1/n}\in\mathbb{R}_{\geq 1}.
\]
It is known that the limit always exists and is also independent of the choice of the ample Cartier divisor and the Weil height function (see \cite[Proposition 12]{KS16a}, \cite[Theorem 3 (a)]{KS16b} for details). Note that \(\alpha_{f^s(x)}=\alpha_f(x)^s\) for any positive integer $s$, just because of the existence of the limit defining $\alpha_f$. Note also that the fundamental inequality \(\alpha_f(x)\leq \delta_f\) holds due to Kawaguchi-Silverman \cite{KS16a} and Matsuzawa \cite{Mat20b}.
\end{enumerate}
\end{definition}

\subsection{Known results on Kawaguchi-Silverman conjecture}\label{subsec: known_results_ksc}
We recall several facts and known results that will be used in this paper.
For more detailed literature review, we refer to \cite{MZ23a} and \cite{Matz23}.

The following lemma is straightforward.
\begin{lemma}\label{lem: ksc iteration}
    Let $f$ be a surjective endomorphism of a projective variety $X$.
    Then KSC holds for $f$ if and only if KSC holds for $f^s$ for some $s>0$.
\end{lemma}

The subsequent lemma facilitates the reduction of the Kawaguchi-Silverman conjecture (KSC, Conjecture \ref{conj: ksc}).

\begin{lemma}[{\cite[Lemma 2.5]{MZ22}}]\label{lem: ksc iff}
Consider the equivariant dynamical systems
$$\xymatrix{
f \acts X \ar@{-->}[r]^{\pi} &Y\racts g
}$$
of normal projective varieties with $\pi$ a dominant rational map.
Then the following hold.
\begin{itemize}
\item[(1)] Suppose $\pi$ is generically finite.
Then KSC holds for $f$ if and only if KSC holds for $g$.
\item[(2)] Suppose $\delta_f=\delta_g$ and KSC holds for $g$.
Then KSC holds for $f$.
\end{itemize}
\end{lemma}

Kawaguchi-Silverman themselves proved KSC for the polarized and abelian cases. Recall that a \emph{$Q$-abelian variety} is a quasi-\'{e}tale quotient of an abelian variety. 

\begin{theorem}[{\cite[Theorem 5]{KS16a}}]\label{thm: ksc polarized}
    KSC holds for polarized endomorphisms.
\end{theorem}

\begin{theorem}[{\cite[Corollary 32]{KS16a}, \cite[Theorem 1.2]{Sil17}, \cite[Theorem 2.8]{MZ22}}]\label{thm: ksc abelian}
    KSC holds for $Q$-abelian varieties.  
\end{theorem}

Recently, Zhong and the first author proved KSC for int-amplified case.
\begin{theorem}[{\cite[Main Theorem]{MZg24}}]\label{thm: ksc int-amplified}
    KSC holds for int-amplified endomorphisms of $\mathbb{Q}$-factorial klt projective varieties.
\end{theorem}

\section{Dynamical Iitaka fibration}\label{sec: dynamical iitaka}

In this section, we recall the recent developed tool introduced by De-Qi Zhang and the first author \cite[Section 4]{MZ23b}.

\begin{definition}[Dynamical Iitaka dimension]\label{def: dynamical Iitaka dimension}
    Let $f$ be a surjective endomorphism of a normal projective variety $X$ and $D$ a $\Q$-Cartier divisor.
    Denote by $V_f(D)$ the subspace of $\Pic_{\Q}(X)$ spanned by $D_i\coloneqq(f^*)^i(D)$ with $i\in\Z$.
    Note that $V_f(D)$ is finite dimensional (cf.~\cite[Proposition 3.7]{MZ22}).
    We define the {\it dynamical $f$-Iitaka dimension} of $D$ as 
    $$\kappa_f(X,D)\coloneqq\max\{\kappa(X,D')\,|\,D'\in V_f(D)\}.$$
    If $D_0,\cdots, D_n$ are effective and span $V_f(D)$, 
    then $\kappa_f(X,D)=\kappa(X, \sum\limits_{i=0}^n D_i)$.
\end{definition}

\begin{lemma}\label{lem: composition formula}
    Let $f$ be a surjective endomorphism of a normal projective variety $X$.
    Then for any $s>0$, we have
    $$R_{f^s}=\sum_{i=0}^{s-1} (f^i)^*R_f.$$
    In particular, if $R_f$ is $\mathbb{Q}$-Cartier, then for $t\gg 1$, we have
    $$\kappa_f(X,R_f)=\kappa(X,R_{f^t}).$$
\end{lemma}

\begin{proof}
    Let $g:X\to X$ be a surjective endomorphism.
    Let $A$ be any prime divisor, $B=f(A)$ and $C=g(B)$.
    Let $b$ be the coefficient of $g^*C$ at $B$ and $a$ the coefficient of $f^*B$ at $A$.
    Then the coefficient of $R_{g\circ f}$ at $A$ is $ab-1$ and the coefficient of $f^*R_g+R_f$ at $A$ is $a(b-1)+a-1=ab-1$.
    Therefore, we have $R_{g\circ f}=f^*R_g+R_f$.
    Substituting $g=f^{s-1}$, we get the desired formula by induction.

    The second statement then just follows from the definition of dynamical Iitaka dimension.  
\end{proof}

\begin{lemma}[{\cite[Lemma 4.2]{MZ23b}}]\label{lem-fkappa>0}
    Let $f$ be a surjective endomorphism of a normal projective variety $X$.
    Let $D$ be a $\Q$-Cartier effective divisor of $X$ such that $\Supp D$ is not $f^{-1}$-periodic.
    Then $\kappa_f(X,D)>0$.
\end{lemma}

\begin{theorem}[{\cite[Theorem 4.6]{MZ23b}}]\label{thm: fiitaka}
    Let $f$ be a surjective endomorphism of a normal projective variety $X$.
    Let $D$ be a $\Q$-Cartier divisor with $\kappa_f(X,D)\ge 0$.
    Then there is an $f$-equivariant dominant rational map $\phi_{f,D}:X\dashrightarrow Y$ to a normal projective variety $Y$
(with $f|_Y$ a surjective endomorphism too)
    satisfying the following conditions.
    \begin{enumerate}
        \item $\dim Y=\kappa_f(X,D)$.
        \item Let $\Gamma$ be the graph of $\phi_{f,D}$. Then the induced projection $\Gamma\to Y$ is equi-dimensional.
        \item $\phi_{f,D}$ is birational to the Iitaka fibration of any $D'\in V_f(D)$ with $\kappa(X,D')=\kappa_f(X,D)$. 
    \end{enumerate}
\end{theorem}

\begin{definition}[{\cite[Definition 4.7]{MZ23b}}]\label{def-fiitaka}
    We call $\phi_{f,D}$ in Theorem \ref{thm: fiitaka} the 
    {\it dynamical $f$-Iitaka fibration} ($f$-Iitaka fibration or dynamical Iitaka fibration for short) of $D$. 
\end{definition}

A special case of the $f$-Iitaka fibration has been studied in \cite[Theorem 7.8]{MZ22}.
\begin{theorem}[{\cite[Theorem 4.8]{MZ23b}}]\label{thm: fiitaka polarized}
    Let $f:X\to X$ be a surjective endomorphism of a projective variety $X$.
    Let $D$ be a $\Q$-Cartier divisor such that $f^*D\equiv qD$ for some $q>1$ and $\kappa(X,D)>0$.
    Let $\phi_{f,D}:X\dashrightarrow Y$ be the $f$-Iitaka fibration of $D$.
    Then $f|_Y$ is $q$-polarized.
\end{theorem}

%%%%
%%%%

\section{Decomposition of cones}\label{sec: decomposition}

In this section, we aim to prove decompositions of the nef cone and the pseudo-effective cone in Theorem \ref{thm: cone}.

\begin{lemma}\label{lem: cone}
    Let $\pi: X \to Y$ be a surjective morphism between normal projective varieties with $\rho(X)=\rho(Y)+1$.
    Suppose there exists a nef and $\pi$-ample divisor $D$ of $X$ with $\nu(X,D)=\dim X-\dim Y>0$. 
    Suppose further either one of the following two assumptions holds.
    \begin{enumerate}
        \item $\Nef(Y)=\PE(Y)$.
        \item $X$ and $Y$ are smooth and $\pi$ is equi-dimensional.
    \end{enumerate}  
    Then we have the decompositions: 
    $$
    \begin{aligned}
        \Nef(X) &= \pi^{\ast}\Nef(Y) \oplus \mathbb{R}_{\geq 0} D, \\ 
        \textup{PE}(X) &= \pi^{\ast}\textup{PE}(Y) \oplus \mathbb{R}_{\geq 0} D.
    \end{aligned}
    $$
\end{lemma}

\begin{proof}[Proof of Lemma \ref{lem: cone} under assumption (1)]
    Note that $\pi^*\Nef(Y)+\mathbb{R}_{\ge 0} D$ is a closed subcone of $\Nef(X)$ and it also spans $\N^1(X)$.
    Then it suffices to show $D+\pi^*B$ is not big for any nef and non-ample $\mathbb{R}$-Cartier divisor $B$ on $Y$.
    Note that $(D+\pi^*B)^{\dim X}=0$ because $D^i\equiv_w0$ for any $i>\dim X-\dim Y$ and $(\pi^*B)^j\equiv_w 0$ for any $j\ge \dim Y$ (cf.~Definition \ref{not: notation}).
\end{proof}

\begin{definition}
    Let $X$ be a smooth projective variety of dimension $n$.
    An $r$-cycle $Z$ on $X$ is said to be nef if 
    $Z\cdot E\ge 0$ for any effective $(n-r)-$cycle $E$.
    Note that an $\mathbb{R}$-Cartier divisor $D$ is pseudo-effective if and only if $D\cdot C\ge 0$ for any nef $1$-cycle $C$.
\end{definition}
\begin{proof}[Proof of Lemma \ref{lem: cone} under assumption (2)] 
    The proof is inspired by \cite{Jow25}. 
    Set $n\coloneqq\dim X, m\coloneqq\dim Y$ and $d\coloneqq n-m$.
    Note that $\pi$ is flat, by the assumption and ``miracle flatness'' (\cite[Theorem 23.1]{Mat86}).
    Since $\rho(X) = \rho(Y)+1$ and $D$ is $\pi$-ample, 
    we have $\N^1(X) = \pi^*\N^1(Y)\oplus \bbr D$.
    
    Let $B \in \Nef(X)$.
    Write $B = \pi^*B_Y + aD$.
    We just need to show that $B_Y \in \Nef(Y)$. 
    Since $\pi$ is flat, for any curve $C \subset Y$, the pullback $\pi^*C$ makes sense and is an effective $(d+1)$-cycle on $X$. 
    Note that $B\cdot D^d\cdot \pi^*C\ge 0$ because $B$ and $D$ are nef.
    Then we have 
    \begin{align*}
        \pi_*(B\cdot D^d\cdot \pi^*C ) &= \pi_*(\pi^*(B_Y\cdot C) \cdot D^d) + \pi_*(\pi^*C \cdot D^{d+1}) \\
        &= B_Y \cdot C \cdot \pi_*(D^d)\ge 0
    \end{align*} 
    by the projection formula.
    In particular, the divisor $B_Y$ is nef.

    Now we let $B \in \PE(X)$ and still write $B = \pi^*B_Y + aD$.
    Let $C$ be a nef $1$-cycle.
    Note that $\pi^*C$ is a nef $(d+1)$-cycle by the projection formula.
    Then $B\cdot D^d\cdot \pi^*C\ge 0$ because
    $B\cdot D^d$ can be realized as the limit of effective cycles $B_i\cdot H_i^d$ where $B_i$ is effective and $H_i$ is ample.  
    So we have 
    \begin{align*}
        \pi_*(B\cdot D^d\cdot \pi^*C ) &= \pi_*(\pi^*(B_Y\cdot C) \cdot D^d) + \pi_*(\pi^*C \cdot D^{d+1}) \\
        &= B_Y \cdot C \cdot \pi_*(D^d) \geq 0
    \end{align*} 
    and hence $B_Y$ is pseudo-effective.
\end{proof}

\begin{remark}
    Using the theory of numerical (co)cycles developed in \cite{Dan20}, 
    we can show that
    \[ \PE(X) = \pi^{\ast}\PE(Y) \oplus \mathbb{R}_{\geq 0} D \]
    without the assumption (2) in Lemma \ref{lem: cone}. Indeed, in the proof of this decomposition, we only have to pull back numerical $(m-1)$-cocycles, which does not require the smoothness of $X, Y$, or the flatness of $\pi$ (see \cite[Theorem 2.3.2]{Dan20}).
\end{remark}

We have the following generalized argument by the same proof of \cite[Proposition 8.4]{MZ22}.
We do not know whether the $\mathbb{Q}$-factorial assumption is necessary or not. However, it is important for the original proof concerning the pullback of weak numerical equivalent class; see also \cite[Lemma 8.3]{MZ22}. 

\begin{proposition}[{cf.~\cite[Proposition 8.4]{MZ22}}]\label{prop: numdim}
	Let $f$ be a surjective endomorphism of a normal projective variety $X$.
	Let $\pi:X\to Y$ be an $f$-equivariant surjective morphism to a $\Q$-factorial normal projective variety $Y$ with $\rho(X)=\rho(Y)+1$ and $\delta_f>\delta_{f|_Y}$.
	Suppose $f^*D\equiv \delta_f D$ for some $\pi$-ample divisor $D$ on $X$.
    Then $D$ is nef and the numerical dimension $\nu(X,D)=\dim X-\dim Y$.
\end{proposition}
\begin{proof}
    We only need to show that $D$ is nef.
    By the assumption, the eigenspace of $\delta_f$ is $1$-dimensional.
    We are done by applying the Perron--Forbenius Theorem (\cite{Bir67}) to $f^*|_{\Nef(X)}$.
\end{proof}

Both the assumptions of Theorems \ref{mainthm: abelian} and \ref{mainthm: rho1} fit into the following theorem.

\begin{theorem}[Cone Decomposition]\label{thm: cone}
    Let $f$ be a surjective endomorphism of a normal projective variety $X$.
	Let $\pi:X\to Y$ be an $f$-equivariant surjective morphism to a $\Q$-factorial normal projective variety $Y$ with $\rho(X)=\rho(Y)+1$ and $\delta_f>\delta_{f|_Y}$.
    Suppose further either one of the following two assumptions holds.
    \begin{enumerate}
        \item $\Nef(Y)=\textup{PE}(Y)$.
        \item $X$ and $Y$ are smooth and $\pi$ is equi-dimensional.
    \end{enumerate}  
    Then we have the decompositions: 
    $$
    \begin{aligned}
        \Nef(X) &= \pi^{\ast}\Nef(Y) \oplus \mathbb{R}_{\geq 0} D, \\ 
        \textup{PE}(X) &= \pi^{\ast}\textup{PE}(Y) \oplus \mathbb{R}_{\geq 0} D,
    \end{aligned}
    $$
    for some nef and $\pi$-ample Cartier divisor $D$ with $f^*D\equiv \delta_f D$.
\end{theorem}

\begin{proof}
    By Lemma \ref{lem: integral}, 
    there exists a nef and $\pi$-ample Cartier divisor $D$ such that $f^*D\equiv \delta_f D$.
    Then $\nu(X,D)=\dim X-\dim Y$ by Proposition \ref{prop: numdim}.
    The two decompositions follow from Lemma \ref{lem: cone}. 
\end{proof}

%%%%
%%%%

\section{Proof of Theorem \ref{mainthm: a character of toric fibraion over k}}\label{sec: character of toric fibration}

In this section, we focus on the proof of Theorem \ref{mainthm: a character of toric fibraion over k}.
The main tool is the sheaf of logarithmic $1$-forms; 
see Section \ref{subsec: forms}.
We also need the following useful result due to De-Qi Zhang.

\begin{proposition}[{\cite[Proposition 2.1]{Zha14}}]\label{prop: totally invariant divisor implies lc pair}
    Let \(X\) be a normal %$($algebraic or analytic$)$
    variety, 
    \(f:X\to X\) a surjective endomorphism of \(\deg(f)>1\) and %\((0\neq)\) 
    $D$ a reduced divisor with \(f^{-1}(D)=D\). 
    Assume:
    \begin{enumerate}
        \item \(X\) is log canonical around \(D\);
        \item \(D\) is \(\mathbb{Q}\)-Cartier; and
        \item \(f\) is ramified around \(D\).
    \end{enumerate}
    Then the pair \((X,D)\) is log canonical around \(D\). 
    In particular, \(D\) is normal crossing outside the union of \(\operatorname{Sing}X\) and a codimension three subset of \(X\).
\end{proposition}

\begin{proof}[Proof of Theorem \ref{mainthm: a character of toric fibraion over k}]
    Denote by $g\coloneqq f|_Y$.
    Set $m\coloneqq \dim X, n\coloneqq \dim Y$ and $d\coloneqq m-n$.
    Note that $D_y\coloneqq D|_{X_y}=D\cap X_y\equiv -K_{X_y}$ is ample and $f^*X_y\equiv (\deg g) X_y$ for general fibre $X_y$ of $\pi$.
    If $d=1$, then $(X_y,D_y)\cong (\mathbb{P}^1, D_y)$ is a toric pair for general $y$.
    So we may assume $d\ge 2$.
    By the projection formula, we have
    $$(\deg f) D^d\cdot X_y=(f^*D)^d\cdot f^*X_y= (q^d\cdot \deg g )D^d\cdot X_y.$$
    So $\deg f=q^d\cdot \deg g$.
   
    By Proposition \ref{prop: totally invariant divisor implies lc pair}, 
    there is a smooth open subset $U\subseteq X$ such that $D \cap U$ is a normal crossing divisor and $\dim X \backslash U  \le \dim X-3$.
    Then $\hat{\Omega}_X^1(\log D)|_U$ is locally free.
    By the ramification divisor formula, 
    we have
    $$K_X+D=f^*(K_X+D)+\Delta_D$$
    where $\Delta_D=R_f-(q-1)D$ is an effective divisor.
    By the assumption $K_X+D$ being $\pi$-trivial, we see that $\Delta_D$ is also $\pi$-trivial.
    In particular, $\Delta_D$ does not dominate $Y$.
    Let $V\subseteq Y\backslash \pi(\Supp \Delta_D)$ such that $X_y$ is smooth, $\dim X_y\backslash U\le \dim X_y-3=\dim X-\dim Y-3$, and $D|_{X_y\cap U}$ is a normal crossing divisor for any $y\in V$.

    \begin{claim}\label{cla: pullback of differential}
        Let $y\in g^{-1}(V)$.
        Then
        $(f^*\hat{\Omega}_{X}^1(\log D))|_{X_y}\cong \hat{\Omega}_{X}^1(\log D)|_{X_y}$.
    \end{claim}
    \begin{proof}
        Note that $f|_{\pi^{-1}(g^{-1}(V))\backslash D}\colon \pi^{-1}(g^{-1}(V))\backslash D\to \pi^{-1}(V)\backslash D$ is \'etale by the purity of branch loci.
        By Lemma \ref{lem: pullback of differential}, we have
        $$(f^*\hat{\Omega}_{X}^1(\log D))|_{\pi^{-1}(g^{-1}(V))} \cong (\hat{\Omega}_{X}^1(\log D))|_{\pi^{-1}(g^{-1}(V))}$$
        and hence the claim follows.
    \end{proof}

    \begin{claim}\label{cla: restriction of ci}
        Let $i=1,2$ and $y\in V$.
        Then
        $$c_i(\hat{\Omega}_X^1(\log D))|_{X_y}=c_i(\hat{\Omega}_{X_y}^1(\log D_y))$$
        and
        for any subsheaf $\mathcal{F}\subseteq \hat{\Omega}_{X_y}^1(\log D_y)$, 
        there exists a subsheaf $\mathcal{G}\subseteq \hat{\Omega}_X^1(\log D)|_{X_y}$ such that $c_i(\mathcal{F})=c_i(\mathcal{G})$.
    \end{claim} 
    
    \begin{proof}
        By Proposition \ref{prop: the second exact sequence},
         we have an exact sequence of locally free sheaves
        $$0\to \mathcal{I}/\mathcal{I}^2|_{U\cap X_y}\xrightarrow{\sigma} \hat{\Omega}_X^1(\log D)|_{U\cap X_y}\xrightarrow{\theta} \hat{\Omega}_{X_y}^1(\log D_y)|_{U\cap X_y}\to 0$$
        where $\mathcal{I}$ is the ideal sheaf of $X_y$ and $\mathcal{I}/\mathcal{I}^2$ is free.
        Let $\mathcal{F}\subseteq \hat{\Omega}_{X_y}^1(\log D_y)$.
        Then there exists $\mathcal{G}\subseteq \hat{\Omega}_X^1(\log D)|_{X_y}$ such that $\Imm \sigma \subseteq \mathcal{G}|_{U\cap X_y}$ and $\theta(\mathcal{G}|_{U\cap X_y})=\mathcal{F}|_{U\cap X_y}$.
        Note that $\dim  X_y\backslash U \le \dim X_y-3$ and $c_i(\mathcal{I}/\mathcal{I}^2)=0$.
        Then we have the desired equalities of first and second Chern classes.
    \end{proof}

    \begin{claim}\label{cla: c1c2=0}
        Let $y\in V\cap g^{-1}(V)$.
        Then we have the Chern classes $c_1(\hat{\Omega}_{X_y}^1(\log D_y))\equiv 0$ and 
        $c_2(\hat{\Omega}_{X_y}^1(\log D_y))\cdot D_y^{d-2}=0$.
    \end{claim}

    \begin{proof}
        Note that $c_1(\hat{\Omega}_{X_y}^1(\log D_y))=K_{X_y}+D_y=(K_X+D)|_{X_y}\equiv 0$ by the assumption.
        We also have
        \begin{align*}
            c_2(\hat{\Omega}_{X_y}^1(\log D_y))\cdot D_y^{d-2}
            &=c_2(\hat{\Omega}_{X}^1(\log D))|_{X_y}\cdot D_y^{d-2}\\
            &=
            \frac{1}{\deg f}f^*c_2(\hat{\Omega}_X^1(\log D))\cdot f^*{X_y}\cdot (f^*D)^{d-2}\\
            &=\frac{1}{q^2}c_2((f^*\hat{\Omega}_X^1(\log D))|_{X_y})\cdot D_y^{d-2}\\
            &=\frac{1}{q^2}c_2(\hat{\Omega}_{X_y}^1(\log D_y))\cdot D_y^{d-2}
        \end{align*}
        where the first equality follows from Claim \ref{cla: restriction of ci}, 
        the second equality follows from the projection formula,  
        the third equality follows from the functoriality of $c_2$, and the last equality follows from Claims \ref{cla: pullback of differential} and \ref{cla: restriction of ci}.
        So the claim is proved.
    \end{proof}

    Given any non-zero torsion-free coherent sheaf $\mathcal{F}$ on $X_y$.
    Denote by 
    $$\mu_{D_y}(\mathcal{F})\coloneqq \frac{c_1(\mathcal{F})\cdot D_y^{d-1}}{\rank \mathcal{F}}.$$
    
    \begin{claim}\label{cla: restriction differential semistable}
        The restriction $\hat{\Omega}_{X}^1(\log D)|_{X_y}$ is $D_y$-slope semistable for general $y$.
    \end{claim}
    \begin{proof}
        By \cite[Theorem 2.3.2 and Remark 2.3.3]{HL10}, 
        after further shrinking $V$, 
        there exists a non-zero subsheaf $\mathcal{F} \subset \hat{\Omega}_X^1(\log D)$,
        such that 
        the restriction $\mathcal{F}|_{X_y}$ is the $D_y$-maximal destabilizing subsheaf of $\hat{\Omega}_X^1(\log D)|_{X_y}$ for $y\in V$.
        Let $y\in g^{-1}(V)\cap V$.
        Note that $f|_{X_y}:X_y\to X_{g(y)}$ is flat because $X_y$ and $X_{g(y)}$ are smooth and $f|_{X_y}$ is finite surjective.
        Then $(f^*\mathcal{F})|_{X_y}=(f|_{X_y})^*(\mathcal{F}|_{X_{g(y)}})$ is still a subsheaf of $$(f|_{X_y})^*(\hat{\Omega}_{X}^1(\log D)|_{X_{g(y)}})\cong (f^*\hat{\Omega}_{X}^1(\log D))|_{X_{y}}\cong \hat{\Omega}_{X}^1(\log D)|_{X_y}$$
        by Claim \ref{cla: pullback of differential}.
        Then 
        \begin{align*}
            \mu_{D_y} (\mathcal{F}|_{X_y})&\ge \mu_{D_y} ((f^*\mathcal{F})|_{X_y})\\
            &=\frac{c_1(f^*\mathcal{F})\cdot X_y\cdot D^{d-1}}{\rank \mathcal{F}}\\
            &=\frac{q\cdot c_1(\mathcal{F})\cdot X_y\cdot D^{d-1}}{\rank \mathcal{F}}\\
            &=q\cdot \mu_{D_y}(\mathcal{F}|_{X_y})
        \end{align*}
        and hence $\mu_{D_y} (\mathcal{F}|_{X_y})=\mu_{D_y}(\hat{\Omega}_{X}^1(\log D)|_{X_y})=0$.
        So the claim is proved.
    \end{proof}

    \begin{claim}\label{cla: fibre differential semistable}
        The sheaf $\hat{\Omega}_{X_y}^1(\log D_y)$ is $D_y$-slope semistable for general $y$.
    \end{claim}

    \begin{proof}
        Let $\mathcal{F}$ be the $D_y$-maximal destabilizing subsheaf of $\hat{\Omega}_{X_y}^1(\log D_y)$.
        By Claim \ref{cla: restriction of ci},
        there exists a subsheaf $\mathcal{G}$ of $\hat{\Omega}_{X}^1(\log D)|_{X_y}$ such that $c_1(\mathcal{F})=c_1(\mathcal{G})$.
        By Claims \ref{cla: c1c2=0} and \ref{cla: restriction differential semistable}, we have 
        $c_1(\mathcal{G})\cdot D_y^{d-1}\le 0$.
        So $\mu_{D_y}(\mathcal{F})\le 0$ and the claim is proved.
    \end{proof}

    \noindent
    {\bf End of Proof of Theorem \ref{mainthm: a character of toric fibraion over k}.}
    Let $y\in Y$ be general.
    Note that the smooth Fano $X_y$ is simply connected (cf. \cite[Corollary 4.18]{Deb01}).
    By Claims \ref{cla: c1c2=0}, \ref{cla: fibre differential semistable} and \cite[Theorem 1.20]{GKP16}, 
    the sheaf 
    $\hat{\Omega}_{X_y}^1(\log D\cap X_y)$ is free.
    In particular, we have the complexity 
    \[ c(X_{y},D\cap X_{y}) \leq \dim X_y + h^1(X_y,\mathcal{O}_{X_y}) - h^0(X_{y}, \hat{\Omega}_{X_y}^1(\log D_y)) = 0\]
    by \cite[Theorem 4.5]{MZ19}.
    Recall Proposition \ref{prop: totally invariant divisor implies lc pair} that $(X,D)$ is log canonical.
    Then $(X_{y},D_{y})$ is still log canonical by \cite[Theorem 5.50]{KM98}.
    It follows from \cite[Theorem 1.2]{BMSZ18} that $(X_{y},D_{y})$ is a toric pair.
\end{proof}

%%%%
%%%%

\section{Positivity of dynamical Iitaka dimension}\label{sec: positivity of ramification divisor}

In this section, we prove Theorem \ref{thm: pos} on the positivity of the dynamical Iitaka dimension of the ramification divisor.

\begin{lemma}\label{lem: wild morphism in codimension 1}
    Let $f$ be a surjective endomorphism of an abelian variety $A$.
    Suppose that $f$ has a Zariski dense orbit.
    Then there exists no $f^{-1}$-invariant prime divisor.
\end{lemma}

\begin{proof}
    Suppose the contrary that $f^{-1}(P)=P$ for some prime divisor $P$.
    By \cite[Lemma 4.6]{Xie25}, $P$ is a translation of an abelian subvariety of $X$.
    We may assume the origin of $A$ is contained in $P$.
    Consider the quotient map
    $$f\acts A\xrightarrow{\pi} A/P\racts g$$
    where $g\coloneqq f|_{A/P}$ is well-defined.
    Note that $\dim A/P=1$ and $g^{-1}(P/P)=P/P$.
    So $g$ is an automorphism of order at most $6$ and hence $f$ has no Zariski dense orbit, a contradiction.
\end{proof}

\begin{lemma}\label{lem: relative picard}
    Let $\pi:X \to Y$ be a surjective projective morphism between $\mathbb{Q}$-factorial normal varieties.
    Suppose that general fibre of $\pi$ is a smooth Fano variety with Picard number $r$.
    Suppose further that $\pi^{-1}(Q)$ is a prime divisor for any prime divisor $Q$ on $Y$.
    Then $\rank \Pic(X)/\pi^*\Pic(Y)\le r$.
\end{lemma}

\begin{proof}
    Let $V$ be a non-empty open subset of $Y$ such that $\pi$ is smooth over $V$.
    Let $D_1, \cdots, D_{r+1}\in \Pic(X)$.
    Note that $X_y\equiv X_{y'}$ for any $y,y'\in V$.
    Then there exist integers $a_i$, not all of which are zero, such that 
    $\sum\limits_{i=1}^{r+1} a_i D_i|_{X_y}\equiv 0$
    for every $y\in V$.
    Denote by $D\coloneqq \sum\limits_{i=1}^{r+1} a_i D_i$ and $\mathcal{L}\coloneqq\pi_*\mathcal{O}_X(D)$.
    For $y\in V$, we have $\Pic^0(X_y)=0$ and 
    $\Pic(X_y)/\Pic^0(X_y)$ is torsion-free because $X_y$ is simply connected (cf.~\cite[Corollary 4.18]{Deb01}).
    Then we have $D|_{X_y}\sim 0$
    for every $y\in V$. 
    By Grauert's Theorem (cf.~\cite[Chapter \uppercase\expandafter{\romannumeral3}, Corollary 12.9 and Exercise 12.4]{Har77}),
    we have $\mathcal{L}|_V\cong \mathcal{O}_V$ and $D|_{\pi^{-1}(V)}\sim 0$.
    So we may assume $\Supp D\subseteq \pi^{-1}(Y\backslash V)$.
    Let $Q_1,\cdots, Q_n$ be prime divisors contained in $Y\backslash V$.
    By the assumption, we have $D\sim \sum\limits_{i=1}^n b_i \pi^{-1}(Q_i)\in \pi^*\Pic_{\Q}(Y)$.
\end{proof}

\begin{theorem}\label{thm: pos}
    Under the assumption of Theorem \ref{mainthm: abelian} or \ref{mainthm: rho1}, we have the dynamical Iitaka dimension $\kappa_f(X,R_f)>0$.
\end{theorem}
\begin{proof}
    Denote by $R_f^h$ the $\pi$-horizontal part of $R_f$.
    It suffices to show $\kappa_f(X,R_f^h)>0$.
    Note that the general fiber of $\pi$ is a smooth Fano variety and the restriction of $f$ on the fibre of $\pi$ has degree greater than $1$.
    Then $R_f^h \neq 0$ by the purity of branch loci and simply connectedness of smooth Fano varieties.
    Suppose the contrary that $\kappa_f(X,R_f^h)=0$.
    Let $D_1,\cdots, D_m$ be irreducible components of $D\coloneqq\Supp R_f^h$.
    Note that $\kappa(X,D_i)=0$ for each $i$.
    By \cite[Lemma 4.2]{MZ23b}, after iteration, 
    we may assume that $f^*D_i= \delta_f D_i$ for each $i$.
    
    We claim that $D$ is a prime divisor.
    Suppose the contrary that $m\ge 2$.
    Then we have $D_1\equiv tD_2$ for some rational number $t>0$ by noting that $\rho(X)=\rho(Y)+1$.
    So 
    $$f^*(D_1-tD_2)=\delta_f(D_1-tD_2).$$
    By \cite[Proposition 6.3]{MMSZZ22}, we have $D_1-tD_2\sim_{\mathbb{Q}} 0$ and hence $\kappa(X,D_1)>0$, a contradiction.

    By Theorem \ref{mainthm: a character of toric fibraion over k}, 
    for general $y \in Y$, $(X_y,D|_{X_y})$ is a toric pair.
    In particular, the number of irreducible components of $D \cap X_y$ is equal to $s\coloneqq \dim X_y + \rho(X_y)>\rho(X_y)$.

    \begin{claim}\label{claim: base change}
        There exist equivariant dynamical systems:  
        $$\xymatrix{
            \save[]+<-1.4pc,0.15pc>*{\widehat{f} \acts} \restore \widehat{X} \ar[r]^{\widehat{\pi}}\ar[d]_{p_X} &\widehat{Y}\ar[d]^{p_Y}\save[]+<1.4pc,0.15pc>*{\racts\widehat{g}} \restore \\
            \save[]+<-1.4pc,0.15pc>*{f \acts} \restore  X \ar[r]^{\pi} &Y\save[]+<1.4pc,0.15pc>*{\racts g} \restore 
        }$$
        such that 
        \begin{enumerate}
            \item $p_Y$ is generically finite,
            \item $\widehat{X}\cong X\times_Y \widehat{Y}$ with $\widehat{\pi}$ and $p_X$ the two natural projections.
            \item $p_X^{-1}(D)$ has exactly $s$ irreducible components $\widehat{D}_1,\cdots,\widehat{D}_s$, and
            \item $\widehat{\pi}$ satisfies the assumption of Lemma \ref{lem: relative picard}.
        \end{enumerate}
    \end{claim}

    Note that base change does not change fibres.
    By Lemma \ref{lem: relative picard}, 
    we have 
    $$\rank \Pic(\widehat{X})/\Pic(\widehat{Y})\le \rho(X_y)<s$$ 
    where $y$ is general.
    Then there exist integers $a_i$, not all of which are zero, such that 
    $$\widehat{D}\coloneqq\sum_{i=1}^s a_i \widehat{D}_i \in \widehat{\pi}^*\Pic(\widehat{Y}).$$
    After iteration, we may assume that $\widehat{f}^*\widehat{D}=q\widehat{D}$.
    By \cite[Proposition 6.3]{MMSZZ22} again, 
    we have $\widehat{D}\sim_{\Q} 0$.
    Therefore, $\kappa(\widehat{X}, \sum\limits_{i=1}^s \widehat{D}_i)>0$.
    By \cite[Theorem 5.13]{Uen75},
    we have
    $\kappa(X,D)=\kappa(\widehat{X}, p_X^*D)>0$,
    a contradiction.
\end{proof}

\begin{proof}[Proof of Claim \ref{claim: base change}]
    By the same strategy as in the proof of \cite[Lemma 5.1]{MZg24},
    there exists a finite surjective morphism $p_Y:\widehat{Y}\to Y$ from a normal projective variety $\widehat{Y}$,
    such that 
    the induced equivariant dynamical systems satisfy (1), (2) and (3).
    
    \smallskip
 
    {\bf Under the assumption of Theorem \ref{mainthm: abelian}}.
    Since $Y$ is an abelian variety, 
    our $g$ is \'etale and hence 
    the branch divisor $B_{p_Y}$ is $g^{-1}$-invariant.
    By Proposition \ref{lem: wild morphism in codimension 1}, 
    we have $B_{p_Y} = 0$ and hence $p_Y$ is \'etale by the purity of the branch loci.
    It follows that $\widehat{X}$ is smooth and $\widehat{Y}$ is an abelian variety.

    Let $\Sigma$ be the subset of $y\in Y$ such that $\pi^{-1}(y)$ is reducible.
    Denote by $\overline{\Sigma}$ the Zariski closure of $\Sigma$.
    Then $g^{-1}(\overline{\Sigma})\subseteq \overline{\Sigma}$
    by \cite[Lemmas 7.2 and 7.4]{CMZ20}.
    So ${g}^{-1}(\overline{\Sigma})= \overline{\Sigma}$.
    Note that $g$ admits a Zariski dense orbit, by Lemma \ref{lem: wild morphism in codimension 1}, the codimension of $\overline{\Sigma}$ in $Y$ is at least $2$.
    Since $p_Y$ is finite, the same statement holds for $\widehat{\pi}:\widehat{X} \to \widehat{Y}$.
    Since $X$ is smooth, the cone theorem \cite[Theorem 3.7]{KM98} implies that $\pi(P)$ is either $Y$ or a prime divisor on $Y$.
    Let $\widehat{Q}$ be a prime divisor on $\widehat{Y}$.
    Then the general fiber of $\widehat{\pi}^{-1}(\widehat{Q}) \to \widehat{Q}$ is irreducible.
    Hence if $\widehat{\pi}^{-1}(\widehat{Q})$ is reducible, there is a component $\widehat{P}$ of $\widehat{\pi}^{-1}(\widehat{Q})$ such that $\widehat{\pi}(\widehat{P})$ is of codimension at least $2$ in $\widehat{Y}$, a contradiction.
    So (4) is satisfied.
    
    \smallskip
    \textbf{Under the assumption of Theorem \ref{mainthm: rho1}}.
    Since $\delta_g=1$, we see that $g$ and hence $\widehat{g}$ are automorphisms.
    Then we can replace $\widehat{Y}$ by its $\widehat{g}$-equivariant resolution of singularities
    by \cite[Theorem 1.1]{BM08}.
    The cost is that $p_Y$ and $p_X$ are generically finite.
    However, the morphism $\widehat{\pi}$ remains smooth.
    In particular, 
    we have that $\widehat{X}$ and $\widehat{Y}$ are now smooth and that every fibre of $\widehat{\pi}$ is irreducible since $\pi$ has connected fibres.
    So (4) is satisfied.
\end{proof}

%%%%
%%%%

\section{Proof of Theorems \ref{mainthm: abelian} and \ref{mainthm: rho1}}\label{sec: proofs of structure theorems}

In this section, we prove the main structure theorems.

\begin{lemma}\label{lem: eigenvector abelian}
    Under the assumption of Theorem \ref{mainthm: abelian}, 
    we have $f^*R_f\equiv \delta_f R_f$.
\end{lemma}

\begin{proof}
    By Lemma \ref{lem: integral}, there exists a nef and $\pi$-ample Cartier divisor $D$ such that $f^*D\sim_{\Q} \delta_fD$.
    Since $-K_X$ is $\pi$-ample, we can write 
    $$K_X\equiv \pi^*B+aD$$
    for some $\Q$-Cartier divisor $B$ on $Y$ and $a<0$.
    We first verify the following claim.

    \begin{claim}\label{cla: nefness}
        The divisor $B$ is nef.
    \end{claim}

    \begin{proof}
        Suppose the contrary that $B$ is not nef.
        Let $\mathbb{R}_{\ge 0}F\subseteq \NE(X)$ be the extremal ray with $\pi_*F=0$.
        Let $\mathbb{R}_{\ge 0}Z\subseteq \NE(X)$ be a $\pi^*B$-negative extremal ray.
        Then $\mathbb{R}_{\ge 0}F$ and $\mathbb{R}_{\ge 0}Z$ are two different rays.
        Note that 
        $$K_X\cdot Z=(\pi^*B+aD)\cdot Z<0$$
        where the inequality is due to $D$ being nef and $a<0$.
        In particular, $\mathbb{R}_{\ge 0}Z$ is a $K_X$-negative extremal ray.

        We notice that the abelian variety $Y$ contains no rational curve.
        By the cone theorem (cf.~\cite[Theorem 3.7]{KM98}), 
        we see that $\pi$ is the only contraction of $K_X$-negative extremal ray.
        So $\pi_*Z=0$ and hence $\pi^*B\cdot Z=0$, a contradiction.        
    \end{proof}

    By the ramification formula, we have 
    \[ R_f = K_X - f^*K_X \equiv \pi^*(B - g^*B) + a(\delta_f-1)D \]
    where $g=f|_Y$.
    Then $B - g^*B$ is pseudo-effective by the Cone Decomposition Theorem \ref{thm: cone}.
    Let $m$ be a positive integer such that $mB$ is Cartier.
    
    It suffices for us to show $\Delta\coloneqq B - g^*B \equiv 0$.
    Fix an ample Cartier divisor $H$ on $Y$.
    Suppose the contrary that $\Delta\not\equiv 0$.
    Then $m (g^i)^*\Delta \cdot H^{\dim X-1}$ is a positive integer for any $i\ge 0$ by noting that $m (g^i)^*\Delta$ is a pseudo-effective Cartier divisor which is not numerically trivial.
    By Claim \ref{cla: nefness}, we have 
    \[ mB \cdot H^{\dim X-1} \geq (mB - m(g^{n+1})^*B)\cdot H^{\dim X-1} = \sum_{i=0}^n m (g^i)^*\Delta \cdot H^{\dim X-1} \geq n\]
    for all $n>0$.
    This is a contradiction.
\end{proof}

\begin{lemma}\label{lem: eigenvector rho1}
    Under the assumption of Theorem \ref{mainthm: rho1}, 
    we have $f^*R_f\equiv \delta_f R_f$.
\end{lemma}

\begin{proof}
    Since $\rho(X)=2$ and $\delta_f>\delta_{f|_Y}=1$, 
    we see that $f^*|_{\N^1(X)}$ is similar to $\diag[\delta_f, 1]$.
    Then the ramification divisor formula implies that 
    $$R_f=K_X-f^*K_X\in \Imm (\id-f^*)|_{\N^1(X)}=\ker(\delta_f\id-f^*)|_{\N^1(X)}.$$
    Therefore, $f^*R_f\equiv \delta_fR_f$.
\end{proof}

\begin{proof}[Proof of Theorems \ref{mainthm: abelian} and \ref{mainthm: rho1}]
    By Lemma \ref{lem: eigenvector abelian} and \ref{lem: eigenvector rho1}, 
    $f^*R_f\equiv\delta_f R_f$.
    By Theorem \ref{thm: pos}, we see that $\kappa_f(X,R_f)>0$.

    We claim that $\kappa_f(X,R_f)<\dim X$.
    Suppose the contrary that $\kappa_f(X,R_f)=\dim X$.
    By Lemma \ref{lem: composition formula}, we have $\kappa(X,R_{f^s})=\kappa_f(X,R_f)$ for some $s\gg 1$.
    Then $R_{f^s}$ is big.
    Note that $f^*R_{f^s}\equiv \delta_f R_{f^s}$ and $\delta_f>1$ is an integer by Lemma \ref{lem: integral}.
    Then $f$ is $\delta_f$-polarized by \cite[Proposition 1.1]{MZ18}.
    Note that $\dim Y>0$.
    By \cite[Lemma 3.1 and Theorem 3.11]{MZ18}, 
    the endomorphism $g$ is also $\delta_f$-polarized.
    This is a contradiction because $\delta_g<\delta_f$.
    So the claim is proved.

    Finally, we are done by Theorems \ref{thm: fiitaka} and \ref{thm: fiitaka polarized}.
\end{proof}

%%%%
%%%%

\section{Proof of Corollaries \ref{maincor: abelian} and \ref{maincor: pn}}\label{sec: proofs of corollaries}
In this section, we prove Kawaguchi-Silverman conjecture for certain fibrations.

\begin{proof}[Proof of Corollary \ref{maincor: abelian}]
    Note that the Fano contraction $\pi\colon X\to Y$ is also the Albanese map of $X$.
    Then $\pi$ is $f$-equivariant.
    Note that $\delta_f\ge \delta_{f|_Y}$.
    By Theorem \ref{thm: ksc abelian}, KSC holds for $f|_Y$.

    Suppose $\delta_f=\delta_{f|_Y}$.
    Then KSC holds for $f$ by Lemma \ref{lem: ksc iff}.

    Suppose $\delta_f>\delta_{f|_Y}$.
    If $\dim Y=0$, then $f$ is $\delta_f$-polarized and KSC holds for $f$ by Theorem \ref{thm: ksc polarized}.
    So we may assume $\dim Y>0$.
    By Theorem \ref{mainthm: abelian}, there is an $f$-equivariant dominant rational map $\varphi:X\dashrightarrow Z$ with $0<\dim Z<\dim X$ and $f|_Z$ is $\delta_f$-polarized.
    Since $\dim Z>0$, we have $\delta_{f|_Z}=\delta_f$.
    By Theorem \ref{thm: ksc polarized}, KSC holds for $f|_Z$.
    By Lemma \ref{lem: ksc iff}, KSC holds for $f$.
\end{proof}

\begin{proof}[Proof of Corollary \ref{maincor: pn}]
    Let $\pi:X\to Y$ be the $\mathbb{P}^n$-bundle.
    By Lemma \ref{lem: ksc iteration}, we are allowed to do iteration.
    After iteration, we may assume $\pi$ is $f$-equivariant (cf.~\cite[Lemma 6.2 and Remark 6.3]{MZ18}).
    Denote by $g\coloneqq f|_Y$
    
    {\bf Case: $Y$ is $Q$-abelian.} 
    We apply the Albanese closure introduced in \cite[Lemma 2.12]{NZ10}; see also \cite[Lemma 8.1]{CMZ20}.
    Let $\sigma:A\to Y$ be the Albanese closure of $Y$.
    Then $g$ lifts to a surjective endomorphism of $A$.
    Let $W\coloneqq X\times_Y A$.
    Then $f$ lifts to a surjective endomorphism of $W$.
    Note that the projection $W\to X$ is finite surjective and the projection $W\to Y$ is still a $\mathbb{P}^n$-bundle.
    By Lemma \ref{lem: ksc iff}, we may replace $X\to Y$ by $W\to A$ and assume $Y$ is an abelian variety.
    Note that $\pi$ is a Fano contraction of an extremal ray.
    Then we are done by Corollary \ref{maincor: abelian}.
    
    {\bf Case: $\rho(Y)=1$}.
    We may assume $\delta_f>1$.
    Suppose $\delta_f=\delta_g$. 
    Then $g$ is $\delta_g$-polarized and KSC holds for $g$ and hence $f$ by Theorem \ref{thm: ksc polarized} and Lemma \ref{lem: ksc iff}.
    Suppose $\delta_f>\delta_g=1$.
    Then we may apply Theorem \ref{mainthm: rho1} and follow the same argument in the proof of Corollary \ref{maincor: abelian}.
    Suppose $\delta_f>\delta_g>1$.
    Then $f^*|_{\N^1(X)}$ is similar to $\diag[\delta_f,\delta_g]$.
    By \cite[Theorem 1.1]{Men20}, we see that $f$ is int-amplified.
    Theorem \ref{thm: ksc int-amplified} concludes the proof.
\end{proof}

\bibliographystyle{alpha}
\bibliography{ref.bib}

\end{document}